\documentclass[12pt]{article}
\usepackage{amssymb,amsmath}
\usepackage{graphicx,float,color}
\usepackage[encapsulated]{CJK}

\usepackage{geometry} % to change the page dimensions
\usepackage{amsthm}
\usepackage{array} % for better arrays (eg matrices) in maths
\usepackage{paralist} % very flexible & customisable lists (eg. enumerate/itemize, etc.)
\usepackage{verbatim} % adds environment for commenting out blocks of text & for better verbatim
\usepackage{subfig} % make it possible to include more than one captioned figure/table in a single float
\usepackage{amsmath}%%%AMS SYMBOLS
\usepackage{mathrsfs}
\usepackage[all]{xy}
\newtheorem{theorem}{Theorem}[section]\newtheorem {lemma}{Lemma}[section]\newtheorem {proposition}{Proposition}[section]
\newtheorem {corollary}{Corollary}[section]
\usepackage{sectsty}
\allsectionsfont{\sffamily\mdseries\upshape} % (See the fntguide.pdf for font help)
\title{${\mathscr F}$-equi\-con\-ti\-nui\-ty and an Analogue of Auslander-Yorke
Dichotomy Theorem}
\author{Hyonhui Ju ${^\dagger}$, Jinhyon Kim${^\dagger}$, Songhun Ri${^\dagger}$, Peter Raith${^\ddagger}$ }
\date{}

\begin{document}
\maketitle
{\makeatletter\renewcommand*\@makefnmark{}\footnotetext{
The first author was supported by the  Ernst Mach Follow-Up Grant (EZA) of OeAD for post-docs in Austria. 

*Corresponding author E-mail address:
jhyonhui@163.com

jh.kim@ryongnamsan.edu.kp }\makeatother}

\centerline{${^\dagger}$Faculty of Mathematics, {\bf Kim Il Sung}
University, Pyongyang, DPR Korea}
\centerline{${^\ddagger}$Faculty of Mathematics, University of Vienna, Austria}

\begin{abstract}

In this paper, we introduce an ${\mathscr
F}$-equi\-con\-ti\-nui\-ty and show an analogue of
Auslander-Yorke dichotomy theorem for ${\mathscr
F}$-sensitivity. Precisely, under the condition that $k{\mathscr
F}$ is translation invariant, we prove that a transitive system is
either ${\mathscr F}$-sensitive or almost $k{\mathscr
F}$-equi\-con\-ti\-nuo\-us , and so generalize the result of
previous work. Also we show that ${\mathscr
F}$-equi\-con\-ti\-nui\-ty is preserved by an open factor map and
consider the implication between ${\mathscr
F}$-equi\-con\-ti\-nui\-ty and mean equi\-con\-ti\-nui\-ty.
\end{abstract}


%%%%%%%%%%%%%%%%% 1. Introduction
\section{Introduction}

A topological dynamical system $(X,\; T)$ means a compact metric
space $(X,\; d)$ with a continuous self-surjection $T$ defined on
it. Throughout this paper we are only interested in a nontrivial
topological dynamical system, where the state space is a compact
metric space without isolated points. Here a trivial dynamical
system means that the state space is a singleton.

A dynamical system $(X,\; T)$ is deterministic in the sense that
the evolution of the system is described by a map $T$, so that the
present(the initial state) completely determines the future(the
forward orbit of the state). Li and Yorke introduced the term
``chaos'' into mathematics in 1975([9]) and showed that a
deterministic system has an unpredictable and complex behavior.
And later many definitions of chaos have been introduced into
mathematics by several scholars, and although there is no
universal mathematical definition of chaos, it is generally agreed
that a chaotic dynamical system should exhibit sensitive
dependence on initial conditions, i.e., minor changes in the
initial state lead to completely different long-term behavior. A
dynamical system $(X,\; T)$ is called \emph{sensitive}([3]) if
there exists $\varepsilon >0$ such that for every $x\in X$ and
every neighborhood $U_{x} $ of $x$, there exist $y\in U_{x} $ and
$n\in {\mathbb N}$ with $d(T^{n} x,\; T^{n} y)>\varepsilon $.

The equi\-con\-ti\-nui\-ty is opposite to the notion of
sensitivity. A dynamical system $(X,\; T)$ is called
\emph{equi\-con\-ti\-nuo\-us} if for every $\varepsilon >0$ there
is a $\delta >0$ such that $d(x,\; y)<\delta $ implies $d(T^{n}
x,\; T^{n} y)<\varepsilon $ for $n=0,\; 1,\; 2,\; \cdots $.

Recall that a point $x\in X$ is \emph{equicontinous }if for every
$\varepsilon >0$ there is a $\delta >0$ such that for every $y\in
X$ with $d(x,\; y)<\delta $, $d(T^{n} x,\; T^{n} y)<\varepsilon $
for $n=0,\; 1,\; 2,\; \cdots $. And recall that a dynamical system
$(X,\; T)$ is called \emph{almost equi\-con\-ti\-nuo\-us} if there
exist some equi\-con\-ti\-nuo\-us points.

The well-known Auslander-Yorke dichotomy theorem([4]) states that
a minimal dynamical system is either sensitive or
equi\-con\-ti\-nuo\-us, which was supplemented in [2]: a
transitive system is either sensitive or almost
equi\-con\-ti\-nuo\-us.

In [8], the authors introduced the notions of mean
equi\-con\-ti\-nui\-ty and mean sensitivity. A dynamical system
$(X,\; T)$ is called \emph{mean equi\-con\-ti\-nuo\-us} if for
every $\varepsilon >0$, there is a $\delta >0$ such that $d(x,\;
y)<\delta $ implies ${\mathop{\lim \sup }\limits_{n\to \infty }}
\frac{1}{n} \sum _{i=0}^{n-1}d(T^{i} x,\; T^{i} y) <\varepsilon $
and a point $x\in X$ is called \emph{mean equi\-con\-ti\-nuo\-us}
if for every $\varepsilon >0$ there is a $\delta >0$ such that for
every $y\in X$ with $d(x,\; y)<\delta $, ${\mathop{\lim \sup
}\limits_{n\to \infty }} \frac{1}{n} \sum _{i=0}^{n-1}d(T^{i} x,\;
T^{i} y) <\varepsilon $. A transitive system $(X,\; T)$ is called  \emph{almost mean equi\-con\-ti\-nuo\-us} if there is at least one  mean equi\-con\-ti\-nuo\-us  point.

A dynamical system $(X,\; T)$ is called \emph{mean sensitive} if
there exists $\varepsilon >0$(sensitive constant) such that for
every $x\in X$ and every neighborhood $U_{x} $ of $x$, there is
$y\in U_{x} $ with ${\mathop{\lim \sup }\limits_{n\to \infty }}
\frac{1}{n} \sum _{i=0}^{n-1}d(T^{i} x,\; T^{i} y) >\varepsilon $.

And they showed an analogue of Auslander-Yorke theorem, which
states that a transitive dynamical system is either almost mean
equi\-con\-ti\-nuo\-us(in the sense of containing some mean
equicontinous points) or mean sensitive and that a minimal system
is either mean equi\-con\-ti\-nuo\-us or mean sensitive.

Recently, some scholars considered various forms of sensitivity
via Furstenberg family such as syndetic sensitivity, cofinite
sensitivity and thickly sensitivity and so on.([6], [10]. [11], [13-15])

In [7] the authors considered equcontinuity via a syndetic
Furstenberg family and introduced a notion of syndetically
equi\-con\-ti\-nui\-ty and showed that an analogue of
Auslander-Yorke dichotomy theorem could also be found for some
stronger forms of sensitivity. Precisely, they proved that a
minimal system is either thickly sensitive or syndetically
equi\-con\-ti\-nuo\-us.1.	Concerning the study on analogue of Auslander-Yorke dichotomy theorem, recently we can find more result in [12] where is obtained an AuslanderYorke’s type dichotomy theorem for  r-sensitivity being  stronger version of sensitivity.

Through the notion of syndetically equi\-con\-ti\-nui\-ty in [7],
we know that it can be generalized to ${\mathscr
F}$-equi\-con\-ti\-nui\-ty, where ${\mathscr F}$ is a Furstenberg
family. Also we know that the notion of mean
equi\-con\-ti\-nui\-ty in [8] is related to ${\mathscr
F}$-equi\-con\-ti\-nui\-ty.

In this paper we introduce an ${\mathscr
F}$-equi\-con\-ti\-nui\-ty and show an analogue of
Auslander-Yorke dichotomy theorem for transitive system(Section
3) which generalize the result of Theorem 3.4 in [7].

 We also show that the notion of mean equi\-con\-ti\-nui\-ty
introduced in [8] could be considered as an ${\mathscr
F}$-equi\-con\-ti\-nui\-ty(Section 4) and that ${\mathscr
F}$-equi\-con\-ti\-nui\-ty is preserved by open factor map.

This paper is organized as follows. In section 2, we provide some
basic concepts and definitions in topological dynamical system. And we introduce the notion of ${\mathscr F}$-equi\-con\-ti\-nui\-ty.
In section 3,  ${\mathscr F}$-equi\-con\-ti\-nui\-ty and an analogue of
Auslander-Yorke dichotomy theorem are discussed.
In section 4, we show that  ${\mathscr F}$-equi\-con\-ti\-nui\-ty is preserved by open factor maps and discuss the implication between mean
equi\-con\-ti\-nui\-ty and ${\mathscr F}$-equi\-con\-ti\-nui\-ty.

%%% 2. Preliminaries
\section{Preliminaries}

\textit{2.1. Furstenberg family}

In this section  we recall some basic concepts related to Furstenberg
family(more detail in [1]).

Denote by ${\mathbb Z}_{+}$ the set of all non-negative integers.

Let ${\mathscr P}$ be the collection of all subsets of ${\mathbb
Z}_{+}$. A collection ${\mathscr F}\subset {\mathscr P}$ is
\emph{Furstenberg family} if $F_{1} \subset F_{2} $ and $F_{1} \in
{\mathscr F}$ imply $F_{2} \in {\mathscr F}$.

Given a Furstenberg family ${\mathscr F}$, define its \emph{dual family}
$k{\mathscr F}$ as follows:

\begin{center}
$k{\mathscr F}=\{ F\in {\mathscr P}:{\mathbb Z}_{+} \backslash
F\notin {\mathscr F}\} =\{ F\in {\mathscr P}:\; $for any $F'\in
{\mathscr F}$, $F\cap F'\ne \emptyset \} $.
\end{center}

Then it is easy to check that ${\mathscr F}$ is a Furstenberg
family if and only if $k{\mathscr F}$is so, and that $k(k{\mathscr
F})={\mathscr F}$.

For $i\in {\mathbb Z}_{+} $ and $F\in {\mathscr P}$, let $F+i=\{
j+i:\; j\in F\}  $ and $F-i=\{ j-i:\; j\in F\}
\cap {\mathbb Z}_{+} $. A Furstenberg family ${\mathscr F}$ is
called \emph{translation invariant} if for any $F\in {\mathscr F}$
and any $i\in {\mathbb Z}_{+} $, $F+i\in {\mathscr F}$ and $F-i\in
{\mathscr F}$.

Given two Furstenberg families ${\mathscr F}_{1} $ and ${\mathscr
F}_{2} $, define
\[{\mathscr F}_{1} \cdot {\mathscr F}_{2} =\{ F_{1} \cap F_{2} :F_{1} \in {\mathscr F}_{1} ,\; F_{2} \in {\mathscr F}_{2} \} .\]

A Furstenberg family ${\mathscr F}$ is said to be a \emph{filter}
if it satisfies ${\mathscr F}\cdot {\mathscr F}\subset {\mathscr
F}$ and it has the \emph{Ramsey property} if $F_{1} \cup F_{2} \in
{\mathscr F}$ implies $F_{1} \in {\mathscr F}$ or $F_{2} \in
{\mathscr F}$.

It can be easily checked that Furstenberg family ${\mathscr F}$
has the Ramsey property if and only if $k{\mathscr F}$ is a
filter.

Let ${\mathscr B}$ be the collection of all infinite subsets of
${\mathbb Z}_{+}$ and ${\mathscr F}_{cf} $ be the family of
cofinite subsets, that is, the collection of subsets of ${\mathbb
Z}_{+}$ with finite complements. It is easy to see that ${\mathscr
F}_{cf} =k{\mathscr B}$.

Let ${\mathscr F}_{t} $ be the collection of the subsets of
${\mathbb Z}_{+}$ which contain arbitrary long runs of positive
integers and denote its dual family by ${\mathscr F}_{s} $. The
element of ${\mathscr F}_{s} $ is called \textit{syndetic} set.
And then the set $F\in {\mathscr P}$ is a syndetic set if and only
if there is an $N\in {\mathbb N}$ such that $\{ i,\; i+1,\; \cdots
,\; i+N\} \cap F\ne \emptyset $ for every $i\in {\mathbb Z}_{+} $.
Also the element of ${\mathscr F}_{t} $ is called \emph{thick}
set.

The set $F\in {\mathscr P}$ is \emph{thickly syndetic} set if for
every $N\in {\mathbb N}$ the positions where length $N$ runs begin
form a syndetic set.

We recall the \emph{upper density} of a set $F\subset {\mathbb
Z}_{+} $ by
\[\overline{D}(F)={\mathop{\lim \sup }\limits_{n\to \infty }} \frac{\# (F\cap [0,\; n-1])}{n} ,\]
where $\# (\cdot )$ means the cardinality of a set([8]). For every
$a\in [0,\; 1)$, set $\overline{D}(a+)=\{ F\in {\mathscr B}:\; \;
\overline{D}(F)>a\} $.

Similarly, $\underline{D}(F)$, the \emph{lower density} of $F$, is
defined by
\[\underline{D}(F)={\mathop{\lim \inf }\limits_{n\to \infty }} \frac{\# (F\cap \{ 0,\; 1,\; \cdots ,\; n-1\} )}{n} .\]

The upper Banach density $BD^{*} (F)$ is defined by
\[BD^{*} (F)={\mathop{\lim \sup }\limits_{N-M\to \infty }} \frac{\# (F\cap [M,\; N])}{N-M+1} .\]

Similarly, we can define the lower Banach density $BD_{*}
(F)$([8]). For every $a\in [0,\; 1)$, set $BD^{*} (a+)=\{ F\in
{\mathscr B}:\; \; BD^{*} (F)>a\} $.

\vskip0.5cm \noindent \textit{2.2. ${\mathscr
F}$-equi\-con\-ti\-nui\-ty}

Firstly  we  recall some concepts of topological dynamical system.

A dynamical system $(X,\; T)$ is called \emph{transitive} if for
any nonempty open subsets $U,\; V\subset X$, $N_{T} (U,\; V)=\{
n\in {\mathbb Z}_{+} :\; U\cap T^{-n} V\ne \emptyset \} $ is
nonempty and a point $x\in X$ is \emph{transitive} if its orbit
$O^{+} (x)=\{ T^{n} (x):n\in {\mathbb Z}_{+} \} $ is dense in $X$.
Denote by ${\rm Tran}(X,\; T)$ the set of all transitive points of
$(X,\; T)$. $(X,\; T)$ is transitive if and only if ${\rm
Tran}(X,\; T)$ is a dense $G_{\delta } $ subset of $X$.

A dynamical system $(X,\; T)$ is called \emph{minimal} if ${\rm
Tran}(X,\; T)=X$.

Let $(X,\; d)$ be a compact metric space and $T:X\to X$ be a
continuous map. And let ${\mathscr F}$ be a Furstenberg family.
For given $x\in X$ and a subset $G\subset X$, set
\[N_{T} (x,\; G)=\{ n\in {\mathbb Z}_{+} :\; T^{n} (x)\in G\} .\]

We write as follows:

\[\Delta _{\varepsilon } =\{ (x,\; y)\in X\times X:\; d(x,\;
y)<\varepsilon \} , B(x,\; \delta )=\{ y\in X:\; d(x,\; y)<\delta
\} ,\]
\[\overline{\Delta }_{\varepsilon } =\{ (x,\; y)\in X\times
X:\; d(x,\; y)\le \varepsilon \} .\]

Now we introduce the notion of ${\mathscr
F}$-equicontinuity.

\vskip0.5cm \noindent{\bf Definition 2.1.} A dynamical system
$(X,\; T)$ is said to be \emph{${\mathscr
F}$-equi\-con\-ti\-nuo\-us} if for every $\varepsilon >0$ there
exists a $\delta >0$ such that whenever $x,\; y\in X$ with $d(x,\;
y)<\delta $,
\[N_{T\times T} ((x,\; y),\; \Delta _{\varepsilon } )\in {\mathscr F}.\]

A point $x\in X$ is called an \emph{${\mathscr
F}$-equi\-con\-ti\-nuo\-us point}(or $(X,\; T)$ is
\emph{${\mathscr F}$-equi\-con\-ti\-nuo\-us at $x\in X$}) if for
every $\varepsilon >0$ there exists a $\delta
>0$ such that for every $y\in B(x,\; \delta )$, $N_{T\times T}
((x,\; y),\; \Delta _{\varepsilon } )\in {\mathscr F}$.

A transitive system is called \emph{almost ${\mathscr
F}$-equi\-con\-ti\-nuo\-us} if there is at least one ${\mathscr
F}$-equi\-con\-ti\-nuo\-us point. The set of all ${\mathscr
F}$-equi\-con\-ti\-nuo\-us points is denoted by ${\rm Eq}_{{\mathscr
F}} (T)$.

Set
\begin{center}
${\rm Eq}_{\varepsilon }^{{\mathscr F}} :=\{ x\in X|\; $there is a
$\delta >0$ such that for any $y,\; z\in B(x,\; \delta )$,
$N_{T\times T} ((y,\; z),\; \Delta _{\varepsilon } )\in {\mathscr
F}\}$.
\end{center}

And we need more following definitions for our study.

\vskip0.5cm \noindent{\bf Definition 2.2.}([6], [10-14]) A
topological dynamical system $(X,\; T)$ is said to be
\emph{${\mathscr F}-$sensitive} if there exists $\varepsilon
>0$(\emph{${\mathscr F}-$sensitive constant}) such that for any
nonempty open subset $U$ of $X$
\[S_{f} (U,\; \varepsilon )=\{ n\in {\mathbb Z}_{+} :\; {\rm diam}\, T^{n} (U)>\varepsilon \} \in {\mathscr F}.\]

In addition, if ${\mathscr F}$ is ${\mathscr F}_{s} $(${\mathscr
F}_{cf} $, ${\mathscr F}_{t} $ respectively), then $(X,\; T)$ is
called \emph{syndetically sensitive}(\emph{cofinitely sensitive},
\emph{thickly sensitive} respectively).

\vskip0.5cm \noindent{\bf Definition 2.3.}([8]) A dynamical system
$(X,\; T)$ is said to be \emph{mean-L-stable} if for every
$\varepsilon >0$ there exists a $\delta >0$ such that whenever
$x,\; y\in X$ with $d(x,\; y)<\delta $,
\[\overline{D}(\{ n\in {\mathbb Z}_{+} :\; d(T^{n} x,\; T^{n} y)\ge \varepsilon \} )<\varepsilon .\]

\vskip0.5cm \noindent{\bf Definition 2.4.}([8]) Let $X$ and $Y$ be
topological spaces and $\pi :X\to Y$ be a map. The map $\pi $ is
called \emph{open} if the image of each nonempty open subset of
$X$ is open in $Y$, and \emph{semi-open} if the image of each
nonempty open subset of $X$ has nonempty interior in $Y$. And $\pi
$ is said to be \emph{open at a point} $x\in X$ if for every
neighborhood $U$ of $x$, $\pi (U)$ is a neighborhood of $\pi (x)$.

\section{An analogue of Auslander-Yorke dichotomy theorem for  ${\mathscr F}$-sensitivity}

In this section we study on analogues of Auslander-Yorke theorem for ${\mathscr F}$-sensitivity using  ${\mathscr F}$-equicontinuity.  

We need following lemmas for it.
%Lemma 2.1.
\begin{lemma}Let $(X,\; T)$ be a dynamical system and
${\mathscr F}$ be a translation invariant Furstenberg family. Then
${\rm Eq}_{\varepsilon }^{{\mathscr F}} $ is an open subset of $X$
and $T^{-1} ({\rm Eq}_{\varepsilon }^{{\mathscr F}} )\subset {\rm
Eq}_{\varepsilon }^{{\mathscr F}} $. Moreover if ${\mathscr F}$ is
a filter then ${\rm Eq}_{{\mathscr F}} (T)=\bigcap _{\varepsilon
>0}{\rm Eq}_{\varepsilon }^{{\mathscr F}}  $.
\end{lemma}

\begin{proof}
Assume that ${\mathscr F}$ is a translation invariant family. Fix
any $x\in {\rm Eq}_{\varepsilon }^{{\mathscr F}} $ and then there
exists a $\delta >0$ such that for any $z,\; w\in B(x,\; \delta
)$, $N_{T\times T} ((z,\; w),\; \Delta _{\varepsilon } )\in
{\mathscr F}$.

If $y\in B(x,\; \delta /2)$ and $z,\; w\in B(y,\; \delta /2)$ then
$z,\; w\in B(x,\; \delta )$.

So $B(x,\; \delta /2)\subset {\rm Eq}_{\varepsilon }^{{\mathscr
F}} $, thus ${\rm Eq}_{\varepsilon }^{{\mathscr F}} $ is an open
set.

Next, if $x\in T^{-1} ({\rm Eq}_{\varepsilon }^{{\mathscr F}} )$
then $Tx\in {\rm Eq}_{\varepsilon }^{{\mathscr F}} $ and there is
a $\delta
>0$ such that for any $y',\; y''\in B(Tx,\; \delta )$,
$N_{T\times T} ((y',\; y''),\; \Delta _{\varepsilon } )\in
{\mathscr F}$.

Since $T$ is continuous, there is a $\eta >0$ such that $y,\; z\in
B(x,\; \eta )$ implies $Ty,\; Tz\in B(Tx,\; \delta )$. So if $y,\;
z\in B(x,\; \eta )$ then $N_{T\times T} ((Ty,\; Tz),\; \Delta
_{\varepsilon } )\in {\mathscr F}$.

Since${\mathscr F}$ is a translation invariant and $1+N_{T\times
T} ((Ty,\; Tz),\; \Delta _{\varepsilon } )\subset N_{T\times T}
((y,\; z),\; \Delta _{\varepsilon } )$,
\[N_{T\times T} ((y,\; z),\; \Delta _{\varepsilon } )\in {\mathscr F}.\]

Therefore $x\in {\rm Eq}_{\varepsilon }^{{\mathscr F}} $ and
$T^{-1} ({\rm Eq}_{\varepsilon }^{{\mathscr F}} )\subset {\rm
Eq}_{\varepsilon }^{{\mathscr F}} $.

Finally, we are going to show that if ${\mathscr F}$ is a filter
then ${\rm Eq}_{{\mathscr F}} (T)=\bigcap _{\varepsilon >0}{\rm
Eq}_{\varepsilon }^{{\mathscr F}}  $.

If $x\in {\rm Eq}_{{\mathscr F}} (T)$ then for any $\varepsilon
>0$ there is a $\delta >0$ such that $y\in B(x,\; \delta )$
implies $N_{T\times T} ((x,\; y),\; \Delta _{\varepsilon /2} )\in
{\mathscr F}$. So for every $y,\; z\in B(x,\; \delta )$, since
\[N_{T\times T} ((y,\; z),\; \Delta _{\varepsilon } )\supset N_{T\times T} ((x,\; y),\; \Delta _{\varepsilon /2} )\cap N_{T\times T} ((x,\; z),\; \Delta _{\varepsilon /2} )\]
and ${\mathscr F}$ is a filter, $N_{T\times T} ((y,\; z),\; \Delta
_{\varepsilon } )\in {\mathscr F}$. Therefore $x\in {\rm
Eq}_{\varepsilon }^{{\mathscr F}} $ and ${\rm Eq}_{{\mathscr F}}
(T)\subset \bigcap _{\varepsilon >0}{\rm Eq}_{\varepsilon
}^{{\mathscr F}}  $.

The proof of ${\rm Eq}_{{\mathscr F}} (T)\supset \bigcap
_{\varepsilon >0}{\rm Eq}_{\varepsilon }^{{\mathscr F}}  $ is
clear. So ${\rm Eq}_{{\mathscr F}} (T)=\bigcap _{\varepsilon
>0}{\rm Eq}_{\varepsilon }^{{\mathscr F}}  $.
\end{proof}

%Lemma 2.2.
\begin{lemma}Let $(X,\; T)$ be a dynamical system and ${\mathscr
F}$ be a filter. Then $(X,\; T)$ is ${\mathscr
F}$-equi\-con\-ti\-nuo\-us if and only if ${\rm Eq}_{{\mathscr F}}
(T)=X$.
\end{lemma}

\begin{proof}
The proof of necessity is clear. We will prove the sufficiency.

Fix any $\varepsilon >0$. Since ${\rm Eq}_{{\mathscr F}} (T)=X$,
for every $x\in X$ there is a $\delta _{x} >0$ such that $y\in
B(x,\; \delta _{x} )$ implies $N_{T\times T} ((x,\; y),\; \Delta
_{\varepsilon /2} )\in {\mathscr F}$.

Given $y,\; z\in B(x,\; \delta _{x} )$, we have
\[N_{T\times T} ((y,\; z),\; \Delta _{\varepsilon } )\supset N_{T\times T} ((x,\; y),\; \Delta _{\varepsilon /2} )\cap N_{T\times T} ((x,\; z),\; \Delta _{\varepsilon /2} )\]
 by the triangular inequality.
And since ${\mathscr F}$ is a filter, we have  $N_{T\times T} ((y,\; z),\;
\Delta _{\varepsilon } )\in {\mathscr F}$.

By the compactness of $X$, there are finite points $x_{1} ,\;
x_{2} ,\; \cdots ,\; x_{N} \in X$ such that $\bigcup
_{i=1}^{N}B(x_{i} ,\; \delta _{x_{i} } /2) =X$. Set $\delta =\min
\{ \delta _{x_{1} } /2,\; \cdots ,\; \delta _{x_{N} } /2\} $.

Now we are going to prove the ${\mathscr
F}$-equi\-con\-ti\-nui\-ty of $(X,\; T)$.

Let $u,\; v\in X$ be two points with $d(u,\; v)<\delta $. Then
there exists an $i\in \{ 1,\; 2,\; \cdots N\} $ such that $u\in
B(x_{i} ,\; \delta _{x_{i} } /2)\subset B(x_{i} ,\; \delta _{x_{i}
} )$. Also $d(u,\; v)<\delta \le \delta _{x_{i} } /2$ implies
$v\in B(x_{i} ,\; \delta _{x_{i} } )$. So $N_{T\times T} ((u,\;
v),\; \Delta _{\varepsilon } )\in {\mathscr F}$ and therefore
$(X,\; T)$ is ${\mathscr F}$-equi\-con\-ti\-nuo\-us.
\end{proof}

Next proposition shows a property of  the set of
${\mathscr F}$-equi\-con\-ti\-nuo\-us points for  a transitive dynamical system.
%Proposition 3.1.
\begin{proposition}
Let $(X,\; T)$ be a transitive dynamical system and ${\mathscr F}$
be a translation invariant Furstenberg family. Then the set of
${\mathscr F}$-equi\-con\-ti\-nuo\-us points is either empty or
residual. If in addition $(X,\; T)$ is almost ${\mathscr
F}$-equi\-con\-ti\-nuo\-us then ${\rm Tran}(X,\; T)\subset {\rm
Eq}_{{\mathscr F}} (T)$. Moreover if in addition ${\mathscr F}$ is
a filter, and $(X,\; T)$ is minimal and almost ${\mathscr
F}$-equi\-con\-ti\-nuo\-us then it is ${\mathscr
F}$-equi\-con\-ti\-nuo\-us.
\end{proposition}

\begin{proof}
By Lemma 2.1, the set ${\rm Eq}_{\varepsilon }^{{\mathscr F}} $ is
open and $T^{-1} ({\rm Eq}_{\varepsilon }^{{\mathscr F}} )\subset
{\rm Eq}_{\varepsilon }^{{\mathscr F}} $. If ${\rm
Eq}_{\varepsilon }^{{\mathscr F}} $ is not empty then for any
nonempty open subset $U$ of $X$, by the transitivity of $(X,\;
T)$, there exists an $n\in {\mathbb Z}_{+} $ such that $\emptyset
\ne T^{-n} ({\rm Eq}_{\varepsilon }^{{\mathscr F}} )\cap U\subset
{\rm Eq}_{\varepsilon }^{{\mathscr F}} \cap U$. So since ${\rm
Eq}_{\varepsilon }^{{\mathscr F}} $ intersects with any nonempty
open subset, ${\rm Eq}_{\varepsilon }^{{\mathscr F}} $ is dense in
$X$.

Thus by the Baire Category theorem, ${\rm Eq}_{{\mathscr F}} (T)$
is empty or residual because ${\rm Eq}_{{\mathscr F}} (T)=\bigcap
_{\varepsilon >0}{\rm Eq}_{\varepsilon }^{{\mathscr F}}  $.

If ${\rm Eq}_{{\mathscr F}} (T)$ is residual then for any
$\varepsilon
>0$ ${\rm Eq}_{\varepsilon }^{{\mathscr F}} $ is open and dense. If
$x\in Tran(X,\; T)$ then there exists $n\in N$ such that $T^{n}
(x)\in {\rm Eq}_{\varepsilon }^{{\mathscr F}} $. So $x\in T^{-n}
({\rm Eq}_{\varepsilon }^{{\mathscr F}} )\subset {\rm
Eq}_{\varepsilon }^{{\mathscr F}} $. Thus $x\in \bigcap
_{\varepsilon
>0}{\rm Eq}_{\varepsilon }^{{\mathscr F}}  ={\rm Eq}_{{\mathscr F}} (T)$.

If $(X,\; T)$ is minimal then ${\rm Tran}(X,\; T)=X$ and so ${\rm
Eq}_{{\mathscr F}} (T)=X$. By Lemma 2.2 $(X,\; T)$ is ${\mathscr
F}$-equi\-con\-ti\-nuo\-us.
\end{proof}

Next  dichotomy theorem and corollary are anologues of the Auslander-Yorke's dichotomy theorem for ${\mathscr F}$-sensitivity.

%Theorem 3.1.
\begin{theorem}
Let $(X,\; T)$ be a transitive dynamical system and ${\mathscr F}$
be a Furstenberg family such that its dual family $k{\mathscr F}$
is translation invariant. Then $(X,\; T)$ is either ${\mathscr
F}$-sensitive and ${\rm Eq}_{k{\mathscr F}} (T)=\emptyset $, or
almost $k{\mathscr F}$-equi\-con\-ti\-nuo\-us and ${\rm Tran}(X,\;
T)\subset {\rm Eq}_{k{\mathscr F}} (T)$.
\end{theorem}

\begin{proof}
It suffices to show that if $(X,\; T)$ is not ${\mathscr
F}$-sensitive then ${\rm Tran}(X,\; T)\subset {\rm Eq}_{k{\mathscr
F}} (T)$. Assume that $(X,\; T)$ is not ${\mathscr F}$-sensitive.
Then for any $\varepsilon >0$ there exists an nonempty open subset
$U$ of $X$ such that $S_{T} (U,\; \varepsilon /2)\notin {\mathscr
F}$. So
\[F=\{ n\in {\mathbb Z}_{+} :\; {\rm diam}T^{n} (U)\le \varepsilon /2\} ={\mathbb Z}_{+} \backslash S_{T} (U,\; \varepsilon /2)\in k{\mathscr F}.\]

Take any $x\in U$ and then there is a $\delta >0$ with $B(x,\;
\delta )\subset U$.

For any $y\in B(x,\; \delta )\subset U$, since $N_{T\times T}
((x,\; y),\; \Delta _{\varepsilon } )\supset N_{T\times T} ((x,\;
y),\; \overline{\Delta }_{\varepsilon /2} )\supset F$, $N_{T\times
T} ((x,\; y),\; \Delta _{\varepsilon } )\in k{\mathscr F}$. Thus
$x\in {\rm Eq}_{k{\mathscr F}} (T)$ and by Proposition 3.1, $(X,\;
T)$ is almost $k{\mathscr F}$-equi\-con\-ti\-nuo\-us and ${\rm
Tran}(X,\; T)\subset {\rm Eq}_{k{\mathscr F}} (T)$.
\end{proof}

\textbf{Remark 1.} Theorem 3.1 coincides with Theorem 3.4 in [7] if
the family ${\mathscr F}$ is replaced with thick family ${\mathscr
F}_{t} $. So Theorem 3.1 is a generalization of Theorem 3.4 in
[7].

%Corollary 3.1.
\begin{corollary}
Assume that ${\mathscr F}$ has Ramsey property and its dual family
$k{\mathscr F}$ is translation invariant. If $(X,\; T)$ is minimal
then it is either ${\mathscr F}$-sensitive or $k{\mathscr
F}$-equi\-con\-ti\-nuo\-us.
\end{corollary}

\textbf{Remark 2.} Corollary 3.1 is a generalization of Auslander-Yorke's theorem which is obtained by replacing the family  ${\mathscr F}$ with the family  ${\mathscr B}$.

\section{The  implication between mean equi\-con\-ti\-nui\-ty and
${\mathscr F}$-equi\-con\-ti\-nui\-ty}

In this section we discuss some relations between mean equi\-con\-ti\-nui\-ty and
${\mathscr F}$-equi\-con\-ti\-nui\-ty.

In [8] is proved that mean equicontinuity is preserved by factor maps.

Here we are going to show that ${\mathscr
F}$-equi\-con\-ti\-nui\-ty is preserved by open factor maps.

%Lemma 3.1.
\begin{lemma} Let $(X,\; T)$ and $(Y,\; S)$ be
topological dynamical systems and $\pi :X\to Y$ be a factor map.
If $x\in X$ is an ${\mathscr F}$-equi\-con\-ti\-nuo\-us point of
$T$ and $\pi $ is open at $x\in X$ then $y=\pi (x)$ is an
${\mathscr F}$-equi\-con\-ti\-nuo\-us point of $S$.
\end{lemma}

\begin{proof}
Since $\pi $ is continuous, for any $\varepsilon
>0$ there is a $\delta >0$ such that $d_{X} (x,\; x')<\delta $
implies $d_{Y} (\pi (x),\; \pi (x'))<\varepsilon $. Here $d_{X} $
and $d_{Y} $ respectively denote the metric of $X$and $Y$. And we
write
\[\Delta _{\delta }^{X} =\{ (x,\; x')\in X\times X:\; d_{X} (x,\; x')<\delta \} , \Delta _{\delta }^{Y} =\{ (y,\; y')\in Y\times Y:\; d_{Y} (y,\; y')<\delta \} .\]

And since $x$ is an ${\mathscr F}$-equi\-con\-ti\-nuo\-us point of
$T$, for the above $\delta >0$ there is a $\delta _{1} >0$ such
that for any $x'\in B(x,\; \delta _{1} )$, $F=N_{T\times T} ((x,\;
x'),\; \Delta _{\delta }^{X} )\in {\mathscr F}$.

So if $n\in F$ then $d_{X} (T^{n} x,\; T^{n} x')<\delta $ and this
implies
\[d_{Y} (\pi (T^{n} x),\; \pi (T^{n} x'))=d_{Y} (S^{n} (\pi (x)),\; S^{n} (\pi (x')))<\varepsilon .\]

Thus $N_{S\times S} ((y,\; \pi (x')),\; \Delta _{\varepsilon }^{Y}
)\supset F$ and this implies $N_{S\times S} ((y,\; \pi (x')),\;
\Delta _{\varepsilon }^{Y} )\in {\mathscr F}$.

Since $\pi $ is open at $x\in X$, $\pi (B(x,\;
\delta _{1} ))$ is a neighborhood of $y=\pi (x)$ and so there is a
$\delta _{2} >0$ with $B(y,\; \delta _{2} )\subset \pi (B(x,\;
\delta _{1} ))$.

Therefore for every $y'\in B(y,\; \delta _{2} )$, $N_{S\times S}
((y,\; y'),\; \Delta _{\varepsilon }^{Y} )\in {\mathscr F}$, that
is, $y$ is an ${\mathscr F}$-equi\-con\-ti\-nuo\-us point of $S$.
\end{proof}

%Theorem 3.2.
\begin{theorem}
Let $(X,\; T)$ and $(Y,\; S)$ be transitive dynamical systems and
$\pi :X\to Y$ be a semi-open factor map. And let ${\mathscr F}$ be
a translation invariant family. If $(X,\; T)$ is almost ${\mathscr
F}$-equi\-con\-ti\-nuo\-us then so is $(Y,\; S)$.
\end{theorem}

\begin{proof}
Since $\pi $ is semi-open, by Lemma 2.1 in [5], the set $\{ x\in
X:\; \pi $ is open at $x\} $ is residual in $X$. So we can take a
transitive point $x\in X$ such that $\pi $ is open at $x\in X$.
Since $(X,\; T)$ is almost ${\mathscr F}$-equi\-con\-ti\-nuo\-us,
by Proposition 3.1 $x\in X$ is an ${\mathscr
F}$-equi\-con\-ti\-nuo\-us point of $T$ and by Lemma 3.1 $y=\pi
(x)$ is also an ${\mathscr F}$-equi\-con\-ti\-nuo\-us point of
$S$. Thus $(Y,\; S)$ is also almost ${\mathscr
F}$-equi\-con\-ti\-nuo\-us.
\end{proof}

Let $(X,\; d_{X} )$ and $(Y,\; d_{Y} )$ be the metric spaces. The
metric on the product space $X\times Y$ is defined by $d((x,\;
y),\; (x',\; y'))=\sqrt{(d_{X} (x,\; x'))^{2} +(d_{Y} (y,\;
y'))^{2} } $. Then the following lemma holds.

%Lemma 3.2.
\begin{lemma}
Let $(X,\; d)$ be a compact metric space and $U$ be a nonempty
open subset of $X\times X$. Let $\Delta ^{X} $ be a diagonal of
$X\times X$, that is, $\Delta ^{X} =\{ (x,\; x):\; x\in X\} $. If
$\Delta ^{X} \subset U$ then there exists a $\delta >0$ such that

\[\Delta _{\delta }^{X} =\{ (x,\; y)\in X\times X:\; d_{X} (x,\; y)<\delta \} \subset U.\]
\end{lemma}

\begin{proof} For every $(x,\; y)\in X\times X$, clearly
$d((x,\; y),\; \Delta ^{X} )=d((y,\; x),\; \Delta ^{X} )$. Since
$\Delta ^{X} $ is closed in $X\times X$, there is a $\delta
>0$ such that
\[B(\Delta ^{X} ,\; \delta )=\{ (x,\; y)\in X\times X:\; d((x,\; y),\; \Delta ^{X} )<\delta \} \subset U.\]

If $(x,\; y)\in \Delta _{\delta }^{X} $ then $\delta >d_{X} (x,\;
y)=d((x,\; y),\; (y,\; y))\ge d((x,\; y),\; \Delta ^{X} )$ and
this implies $(x,\; y)\in B(\Delta ^{X} ,\; \delta )$. Therefore
$\Delta _{\delta }^{X} \subset B(\Delta ^{X} ,\; \delta )\subset
U$.
\end{proof}

%Theorem 3.3
\begin{theorem}
Let $(X,\; T)$ and $(Y,\; S)$ be topological dynamical systems and
$\pi :X\to Y$ be an open factor map. And let ${\mathscr F}$ be a
Furstenberg family. If $(X,\; T)$ is ${\mathscr
F}$-equi\-con\-ti\-nuo\-us then so is $(Y,\; S)$.
\end{theorem}

\begin{proof}
Since $\pi $ is continuous, for every $\varepsilon
>0$ there is a $\delta >0$ such that $d_{X} (x,\; x')<\delta $
implies $d_{Y} (\pi (x),\; \pi (x'))<\varepsilon $. For the above
$\delta>0$, since $(X,\; T)$ is ${\mathscr
F}$-equi\-con\-ti\-nuo\-us, there exists a $\delta _{1} >0$ such
that $d_{X} (x,\; x')<\delta _{1} $ implies $F=N_{T\times T}
((x,\; x'),\; \Delta _{\delta }^{X} )\in {\mathscr F}$. So if
$n\in F$ then $d_{X} (T^{n} x,\; T^{n} x')<\delta $ and this
implies

\[d_{Y} (\pi (T^{n} x),\; \pi (T^{n} x'))=d_{Y} (S^{n}
(\pi (x)),\; S^{n} (\pi (x')))<\varepsilon.\]

Thus $N_{S\times S} ((\pi (x),\; \pi (x')),\; \Delta _{\varepsilon
}^{Y} )\supset F$.

Since $\pi $ is open, $\pi \times \pi :X\times X\to Y\times Y$ is
also open. So $\pi \times \pi (\Delta _{\delta _{1} }^{X} )$ is
open in $Y\times Y$ and contains a diagonal of $Y\times Y$, that
is, $\pi \times \pi (\Delta _{\delta _{1} }^{X} )\supset \Delta
^{Y} $. Then by Lemma 3.2 there exists a $\delta _{2} >0$ such
that $\Delta _{\delta _{2} }^{Y} \subset \pi \times \pi (\Delta
_{\delta _{1} }^{X} )$.

Thus for every $(y,\; y')\in \Delta _{\delta _{2} }^{Y} $ there
exists a pair $(x,\; x')\in \Delta _{\delta _{1} }^{X} $ such that
$y=\pi (x),\; y'=\pi (x')$. Since $d_{X} (x,\; x')<\delta _{1} $,
$N_{S\times S} ((y,\; y'),\; \Delta _{\varepsilon }^{Y} )\supset
F$ and this implies

\[N_{S\times S} ((y,\; y'),\; \Delta _{\varepsilon }^{Y} )\in
{\mathscr F}.\]

Therefore $(Y,\; S)$ is also ${\mathscr
F}$-equi\-con\-ti\-nuo\-us.
\end{proof}

%Lemma 4.1.

Following lemma 4.3 and 4.4  show some implications between mean equicontinuity and  ${\mathscr
F}$-equi\-con\-ti\-nuity.

\begin{lemma}
If $(X,\; T)$ is mean equi\-con\-ti\-nuo\-us then for any $a\in
(0,\; 1)$, it is $k\overline{D}(a+)$-equi\-con\-ti\-nuo\-us. Also
if $(X,\; T)$ is  almost mean equi\-con\-ti\-nuo\-us, then
for any $a\in (0,\; 1)$, it is almost
$k\overline{D}(a+)$-equi\-con\-ti\-nuo\-us.
\end{lemma}

\begin{proof}
Assume that there exists an $a\in (0,\; 1)$ such that $(X,\; T)$
is not $k\overline{D}(a+)$-equi\-con\-ti\-nuo\-us. Then there is
$\varepsilon _{0} >0$ such that for any $\delta >0$, there exist
$x,\; y\in X$ with $d(x,\; y)<\delta $ such that $N_{T\times T}
((x,\; y),\; \Delta _{\varepsilon _{0} /a} )\notin
k\overline{D}(a+)$. So

\begin{center}
${\mathbb Z}_{+} \backslash N_{T\times T} ((x,\; y),\; \Delta
_{\varepsilon _{0} /a} )=\{ n\in {\mathbb Z}_{+} :\; d(T^{n} x,\;
T^{n} y)\ge \varepsilon _{0} /a\} \in \overline{D}(a+)$.
\end{center}

Set $F=\{ n\in {\mathbb Z}_{+} :\; d(T^{n} x,\; T^{n} y)\ge
\varepsilon _{0} /a\} $, and then
\[\begin{array}{l} {{\mathop{\lim \sup }\limits_{n\to \infty }} \frac{1}{n} \sum _{i=0}^{n-1}d(T^{i} x,\; T^{i}
y)=} \\ {\; \; \; \; \; \; \; \; \; \; \; \; \; \; \; \; \; \; =
{\mathop{\lim \sup }\limits_{n\to \infty }} \frac{1}{n}
\left({\mathop{\sum}\limits_{i\in [0,\; n-1]\cap F}}d(T^{i} x,\;
T^{i} y) +{\mathop{\sum}\limits_{i\in [0,\; n-1]\cap F^{C}} }
d(T^{i} x,\; T^{i} y) \right)} \\ {\; \; \; \; \; \; \; \; \; \;
\; \; \; \; \; \; \; \; \ge {\mathop{\lim \sup }\limits_{n\to
\infty }} \frac{\# ([0,\; n-1]\cap F)}{n} \cdot \frac{\varepsilon
_{0} }{a} \ge \varepsilon _{0} .} \end{array}\]

This contradicts to the mean equi\-con\-ti\-nui\-ty of $(X,\; T)$.
The proof of second part is similar to this. 

\end{proof}

%Lemma 4.2.
\begin{lemma}
If $(X,\; T)$ is $k\overline{D}(0+)$-equi\-con\-ti\-nuo\-us then
it is mean equi\-con\-ti\-nuo\-us. Also if $(X,\; T)$ is almost 
$k\overline{D}(0+)$-equi\-con\-ti\-nuo\-us then it is almost 
mean equi\-con\-ti\-nuo\-us.
\end{lemma}

\begin{proof}
Since $(X,\; T)$ is $k\overline{D}(0+)$-equi\-con\-ti\-nuo\-us,
for every $\varepsilon >0$ there exists a $\delta >0$ such that
whenever $x,\; y\in X$ with $d(x,\; y)<\delta $, $N_{T\times T}
((x,\; y),\; \Delta _{\varepsilon } )\in k\overline{D}(0+)$, that
is,
\[\{ n\in {\mathbb Z}_{+} :\; d(T^{n} x,\; T^{n} y)\ge \varepsilon \} ={\mathbb Z}_{+} \backslash N_{T\times T} ((x,\; y),\; \Delta _{\varepsilon } )\notin \overline{D}(0+).\]

So $\overline{D}(\{ n\in {\mathbb Z}_{+} :\; d(T^{n} x,\; T^{n}
y)\ge \varepsilon \} )=0<\varepsilon $, that is, $(X,\; T)$ is
mean-L-stable.

Thus by Lemma 3.1 in [8], $(X,\; T)$ is mean
equi\-con\-ti\-nuo\-us. The proof of second part is similar to this. 
\end{proof}

Following lemma shows an  implication between mean sesitivity and  ${\mathscr
F}$-sesitivity. 
%Lemma 4.3.
\begin{lemma}
If $(X,\; T)$ is mean sensitive with sensitive constant $\delta
>0$ then for any $a\in \left[0,\; \frac{\delta }{{\rm diam}(X)}
\right)$, it is $\overline{D}(a+)$-sensitive.
\end{lemma}

\begin{proof}
For any $a\in \left[0,\; \frac{\delta }{{\rm diam}(X)} \right)$,
set $\delta '=\delta -a\cdot {\rm diam}(X)>0$. Now we are going to
show that $(X,\; T)$ is $\overline{D}(a+)$-senstive with sensitive
constant $\delta '>0$.

Fix any nonempty open subset of $X$ and set $F=S_{T} (U,\; \delta
')$. We choose $x\in U$ and then by the mean continuity of $(X,\;
T)$, there exists $y\in U$ such that

\begin{center}
${\mathop{\lim \sup }\limits_{n\to \infty }} \frac{1}{n} \sum
_{i=0}^{n-1}d(T^{i} x,\; T^{i} y)
>\delta $.
\end{center}

So

\[\begin{array}{l} {{\mathop{\delta <\lim \sup }\limits_{n\to \infty }} \frac{1}{n} \sum _{i=0}^{n-1}d(T^{i} x,\; T^{i} y) } \\
{\; \; ={\mathop{\lim \sup }\limits_{n\to \infty }} \frac{1}{n}
\left({\mathop{\sum}\limits_{i\in [0,\; n-1]\cap F}}d(T^{i} x,\;
T^{i} y) +{\mathop{\sum}\limits_{i\in [0,\; n-1]\cap F^{C}}}
d(T^{i} x,\; T^{i} y) \right)} \\ {\; \; \le {\mathop{\lim \sup
}\limits_{n\to \infty }} \frac{\# ([0,\; n-1]\cap F)}{n} \cdot
{\rm diam}(X)+\delta '}
\\ {\; \; =\overline{D}(F)\cdot {\rm diam}(X)+\delta '.}
\end{array}\]

Therefore $\overline{D}(F)>\frac{\delta -\delta '}{{\rm diam}(X)}
=a$, that is $F=S_{T} (U,\; \delta ')\in \overline{D}(a+)$.
\end{proof}

As a consequence of the above consideration, the following proposition holds. It is immediately followed by Theorem 5.4 in [8].
\begin{proposition}
Let $(X,\; T)$ be a topological dynamical system. If $(X,\; T)$ is
transitive, then there is an $a>0$ such that $(X,\; T)$ is either
$\overline{D}(a+)$-sensitive or
$k\overline{D}(a+)$-equi\-con\-ti\-nuo\-us.
\end{proposition}

\vskip0.5cm\noindent {\bf References}
\small\vskip0.4cm
\begin{itemize}

\item[{[1]}] E. Akin, \textit{Recurrence in topological dynamics,
Furstenberg families and Ellis actions}, New York, Plenum Press;
1997.

\item[{[2]}] E. Akin, J. Auslander, K. Berg, \textit{When is a
transitive map chaotic?}, Convergence in ergodic theory and
probability (Columbus, OH, 1993), Ohio State Univ. Math. Res.
Inst. Publ., 5(1996) 25--40.

\item[{[3]}] E. Akin, S. Kolyada, \textit{Li--Yorke sensitivity},
Nonlinearity 16(2003) 1421--1433.

\item[{[4]}] J. Auslander, J. Yorke, \textit{Interval maps,
factors of maps, and chaos}, T\^{}hoku Math. J. (2) 32 (1980), no.
2, 177--188.

\item[{[5]}] E. Glasner, \textit{The structure of tame minimal
dynamical systems}. Ergod. Th. \& Dynam. Sys. 27 (2007)
1819--1837.

\item[{[6]}] W. Huang, D. Khilko, S. Kolyada, G. Zhang,
\textit{Dynamical compactness and sensitivity}, J.Differ.
Equations. 260(2016) 6800--6827.

\item[{[7]}] W. Huang, S. Kolyada, G. Zhang, \textit{Analogues of
Auslander-Yorke theorems for multi-sensitivity},
arXiv:1509.08818v12 [math.DS] 20 May 2016

\item[{[8]}] J. Li, S. Tu, X. Ye, \textit{Mean
equi\-con\-ti\-nui\-ty and mean sensitivity}, Ergod. Th. \& Dynam.
Sys. 35(2015) 2587 - 2612

\item[{[9]}] T. Li, J. Yorke, \textit{Period three implies chaos},
Am. Math. Monthly. 82(1975) 985--92.

\item[{[10]}] R. Li, \textit{A note on stronger forms of
sensitivity for dynamical systems}, Chaos. Soliton. Fract.
45(2012) 753--758.

\item[{[11]}] R. Li, Y. Zhao, H. Wang, R. Jiang, H. Liang,
\textit{${\mathscr F}$-sensitivity and (${\mathscr F}_{1} $,
${\mathscr F}_{2} $)-sensitivity between dynamical systems and
their induced hyperspace dynamical systems}, J. Nonlinear Sci.
Appl. 10(2017) 1640--1651.

\item[{[12]}]K. Liu,  X. Zhang, 
\textit{Auslander Yorke type dichotomy theorems for stronger version r-sensitivity}, arXiv:1905.08104v1 [math.DS] 20 May 2019.

\item[{[13]}] F. Tan, R. Zhang, \textit{On ${\mathscr
F}$-sensitive pairs}, Acta. Math. Sci. 31B(4)(2011) 1425--1435.

\item[{[14]}] X. Wu, R. Li, Y. Zhang, \textit{The multi-${\mathscr
F}$-sensitivity and (${\mathscr F}_{1} $, ${\mathscr F}_{2}
$)-sensitivity for product systems}, J. Nonlinear Sci. Appl.
9(2016) 4364--4370.

\item[{[15]}] X. Wu, J. Wang, G. Chen, \textit{${\mathscr
F}$-sensitivity and multi-sensitivity of hyperspatial dynamical
systems}, J. Math. Anal. Appl. 429(2015) 16--26.
\end{itemize}

\end{document}